\documentclass{amsart}
\usepackage{amsfonts,amssymb}
\usepackage{enumerate,graphicx}
\usepackage{caption}
\usepackage{subcaption}
\usepackage[square,numbers]{natbib}
\usepackage{mathrsfs,bbm}
\usepackage{tikz,tikzscale,tikz-cd}
\usepackage{multirow,xparse,fp}
\usepackage{environ}
\usepackage{amstext}
\usepackage{hyperref}
\usepackage{diagbox}
\usepackage[pagewise]{lineno}
\usepackage{comment}
\newtheorem{theorem}{Theorem}[section]

\newtheorem{proposition}[theorem]{Proposition}

\newtheorem{lemma}[theorem]{Lemma}
\theoremstyle{definition}

\numberwithin{equation}{section}
\allowdisplaybreaks[4]

\begin{document}
	\title{Hausdorff dimensions of affine multiplicative subshifts}
	
	\author[Jung-Chao Ban]{Jung-Chao Ban}
	\address[Jung-Chao Ban]{Department of Mathematical Sciences, National Chengchi University, Taipei 11605, Taiwan, ROC.}
	\address{Math. Division, National Center for Theoretical Science, National Taiwan University, Taipei 10617, Taiwan. ROC.}
	\email{jcban@nccu.edu.tw}
	
	\author[Wen-Guei Hu]{Wen-Guei Hu}
	\address[Wen-Guei Hu]{College of Mathematics, Sichuan University, Chengdu, 610064, P. R. China}
	\email{wghu@scu.edu.cn}
		
	\author[Guan-Yu Lai]{Guan-Yu Lai}
	\address[Guan-Yu Lai]{Department of Mathematical Sciences, National Chengchi University, Taipei 11605, Taiwan, ROC.}
	\email{gylai@nccu.edu.tw}

 \author[Lingmin Liao]{Lingmin Liao}
	\address[Lingmin Liao]{School of Mathematics and Statistics, Wuhan University, Wuhan, Hubei 430072, P. R. China}
	\email{lmliao@whu.edu.cn}

	
	\thanks{}
	
	
	\baselineskip=1.2\baselineskip
	
	\begin{abstract}
		 We calculate the Minkowski and Hausdorff dimensions of affine multiplicative subshifts on $\mathbb{N}$.     
	\end{abstract}
	\maketitle
 \section{Introduction}

Pioneered by Furstenberg with his seminal work on the ergodic proof of the
Szemer\'edi Theorem and subsequent extensions (cf. \cite{ furstenberg1982ergodic,furstenberg1978topological}), the \emph{multiple ergodic theory} is a significant research topic in the field of dynamical systems and ergodic theory. Numerous applications can be found in the fields of number theory, combinatorics, and statistical physics (cf. \cite{bourgain1990double,frantzikinakis2011some,
frantzikinakis2023multiple,furstenberg2011mean, host2005nonconventional,
host2018nilpotent, krause2022pointwise}). In view of fractal analysis, Fan, Liao and Ma \cite{fan2012level} obtained the Minkowski dimension of the set 
\[
X^{(2)}=\left\{(x_i)_{i=1}^\infty\in \{0,1\}^{\mathbb{N}}:x_{i}x_{2i}=0\text{ for all }i\in\mathbb{N}\right\}\text{.}
\]
Their result was later generalized by Kenyon, Peres and Solomyak \cite{kenyon2012hausdorff} who calculated both the Hausdorff and Minkowski dimensions of the so-called \emph{multiplicative shift of finite type }(SFT)  
\begin{equation}\label{1}
X_{A}^{(q)}=\left\{(x_i)_{i=1}^\infty\in\Sigma_m^{\mathbb{N}}:A(x_{i},x_{qi})=1\text{
for all }i\in \mathbb{N}\right\}\text{,}  
\end{equation}
where $q\geq 2$ is an integer, $\Sigma_m=\{0,1,...,m-1\}$ and $A$ is an $m\times m$ binary matrix. Precisely, they proved the following theorem.
\begin{theorem}[Kenyon et al. \cite{kenyon2012hausdorff}, Theorem 1.3]
\label{Thm: 2}Let $q\geq 2$ be an integer and $A$ be an $m\times m$ primitive binary matrix.
\begin{enumerate}
\item The Minkowski dimension of $X_{A}^{(q)}$ exists and equals 
\begin{equation}\label{5}
\dim _{M}X_{A}^{(q)}=(q-1)^{2}\sum_{i=1}^{\infty }\frac{\log _{m}\left\vert
A^{i-1}\right\vert }{q^{i+1}}\text{,}  
\end{equation}
where $\left\vert A\right\vert $ stands for the sum of all entries of the
matrix $A$.
\item The Hausdorff dimension of $X_{A}^{(q)}$ is given by 
\begin{equation}\label{6}
\dim _{H} X_{A}^{(q)}=\frac{q-1}{q}\log _{m}\sum_{i=0}^{m-1}t_{i}\text{,}
\end{equation}
where $(t_{i})_{i=0}^{m-1}$ is the unique positive vector satisfying $%
t_{i}^{q}=\sum_{j=0}^{m-1}A(i,j)t_{j}$.
\end{enumerate}
\end{theorem}

If the configuration on the positions $\left( iq^{\ell}\right) _{\ell=0}^{\infty }$
for all $i\in \mathbb{N}$, is specified by some subshift $\Omega $ of $\mathbb{N}$, that is, 
\[
X_{\Omega }^{(q)}=\left\{(x)_{i=1}^\infty\in \Sigma_m^{\mathbb{N}}:\left(x_{iq^{\ell}}\right)_{\ell=0}^{\infty }\in \Omega \text{ for all }i\in\mathbb{N}\right\},
\]
then we call it \emph{multiplicative subshift}. The name `multiplicative' comes from the fact that $X_{\Omega }^{(q)}$ is invariant under the action of
multiplicative integers. That is, if $(x_{i})_{i=1}^{\infty }\in X_{\Omega
}^{(q)}$, then $(x_{ki})_{i=1}^{\infty }\in X_{\Omega }^{(q)}$ for all $k\in 
\mathbb{N}$. Numerous investigations have been conducted regarding the complicity of the multiplicative subshifts since then. For example, Peres, Schmeling, Seuret and Solomyak \cite{peres2014dimensions} obtained
the Hausdorff and Minkowski dimensions of the \emph{3-multiple multiplicative subshifts}:
\begin{equation}\label{2}
X_{\Omega }^{(2,3)}=\left\{(x_i)_{i=1}^\infty\in \Sigma_m^{\mathbb{N}}:x_{i}x_{2i}x_{3i}=0\text{ for all }i\in \mathbb{N}\right\}\text{.}  
\end{equation}
We note that the set $X_{\Omega }^{(2,3)}$ can be written in the following generalized form:
\[
X_{\Omega }^{(S)}=\left\{(x_i)_{i=1}^\infty\in \Sigma_m^{\mathbb{N}}:x|_{iS}\in \Omega \mbox{ for all }i\in \mathbb{N} \mbox{ with }\gcd (i,S)=1\right\},
\]
where $\gcd (i,S)=1$ means that $\gcd (i,s)=1$, $\forall s\in S$, and $x|_{iS}:=(x_{i\cdot s})_{s\in S}$. Ban, Hu and Lin \cite{ban2019pattern} derived the entropy formula (or the Minkowski dimension) for the general $r$-multiple multiplicative subshifts for $r\in \mathbb{N}$. Fan \cite{fan2021multifractal}, Schmeling and Wu \cite{fan2016multifractal} studied the multifractal analysis of general multiple ergodic averages. Brunet \cite{brunet2021dimensions}, and Ban, Hu and Lai \cite{ban2021entropy,ban2023hausdorff} investigated the Hausdorff and Minkowski dimensions of the multidimensional multiplicative subshifts.

The present paper aims to investigate the \emph{affine multiplicative shift of finite type (SFT)}. Precisely, let $A$ be an $m\times m$ binary matrix, $1\leq p<q\in \mathbb{N}$ and $a,b\in 
\mathbb{Z}$. Define an affine multiplicative shift by
\begin{equation}
X_{A}^{(p,q;a,b)}=\left\{(x_i)_{i=1}^\infty\in \Sigma_m^{\mathbb{N}}:\text{ }%
A\left(x_{pk+a},x_{qk+b}\right)=1\text{ for all }k\geq 1\right\}\text{.}  \label{3}
\end{equation}
It is easily seen that $X_{A}^{(q)}$ in (\ref{1}) is a special case of (\ref{3}) by taking $(p,q;a,b)=(1,q;0,0)$. 

We calculate the Hausdorff and Minkowski dimensions of the set $X_{A}^{(p,q;a,b)}$.

\begin{theorem}[Minkowski dimension]
\label{Thm: 1}Let $A$ be an $m\times m$ irreducible binary matrix, $1\leq
p<q $ and $a,b\in \mathbb{Z}$.

\begin{enumerate}
\item If $(p,q)\nmid (b-a)$, then 
\[
\dim _{M}X_{A}^{(p,q;a,b)}=1-\frac{2}{q}+\frac{1}{q}\log _{m}\left\vert
A\right\vert .
\]
\item If $(p,q)\mid (b-a)$, then 
\[
\dim _{M}X_{A}^{(p,q;a,b)}=1-\frac{2q_{1}-1}{(p,q)q_{1}^{2}}+\frac{%
(q_{1}-1)^{2}}{(p,q)}\sum_{i=2}^{\infty }\frac{\log _{m}\left\vert
A^{i-1}\right\vert }{q_{1}^{i+1}}\text{,} 
\]
where $q_{1}=\frac{q}{(p,q)}$. In particular, if $(p,q)=1$, we have 
\begin{equation}
\dim _{M}X_{A}^{(p,q;a,b)}=(q-1)^{2}\sum_{i=1}^{\infty }\frac{\log
_{m}\left\vert A^{i-1}\right\vert }{q^{i+1}}\text{.}  \label{7}
\end{equation}
\end{enumerate}
\end{theorem}

We remark that if $(p,q,a,b)=(1,q,0,0)$, then by (\ref{7}), we obtain (\ref{5}), i.e., the result of Kenyon et al. on the Minkowski dimension of $X_{A}^{(q)}$.

\begin{theorem}[Hausdorff dimension]\label{Thm: 3}
Let $A=[a_{i,j}]_{i,j=0}^{m-1}$ be an $m\times m$ irreducible binary matrix.
The following assertions hold true.

\begin{enumerate}
\item If $1\leq p<q$, $a$, $b\in \mathbb{Z}$ and $(p,q)\nmid (b-a)$, then 
\[
\dim _{H}X_{A}^{(p,q;a,b)}=1-\frac{1}{p}-\frac{1}{q}+\frac{1}{p}\log
_{m}\sum_{i=0}^{m-1}a_{i}{}^{\frac{p}{q}}\text{,} 
\]
where $a_i=\sum_{j=0}^{m-1} a_{i,j}$.
\item If $1<p<q$, $a$, $b\in \mathbb{Z}$ and $(p,q)\mid (b-a)$, then 
\[
\dim _{H}X_{A}^{(p,q;a,b)}=1-\frac{1}{p}-\frac{1}{q}+\frac{1}{(p,q)p_{1}q_{1}%
}+\sum_{i=2}^{\infty }\left( \sum_{j=1}^{i}P_{ij}\right) \log
_{m}t_{\phi ;i}\text{,} 
\]
where $p_1(p,q)=p$, $P_{i,j}$ and $t_{\phi;i}$ are defined in (\ref{eq 3.1-1}) and (\ref{eq 3.1-2}).
\item If $1=p<q$, $a,b\in \mathbb{Z}$ and $A$ is primitive, then 
\begin{equation}
\dim _{H}X_{A}^{(p,q;a,b)}=\frac{q-1}{q}\log _{m}\sum_{i=0}^{m-1}t_{i},\label{8}
\end{equation}
where $(t_{i})_{i=0}^{m-1}$ is the unique positive vector satisfying $%
t_{i}^{q}=\sum_{j=0}^{m-1}a_{i,j}t_{j}$.

\item We have $\dim _{H}X_{A}^{(p,q;a,b)}\leq \dim _{M}X_{A}^{(p,q;a,b)}$ and the equality holds only when the row sums $a_i$ of $A$ are all equal.
\end{enumerate}
\end{theorem}
We remark that the condition $1\leq p$ in (1) of Theorem \ref{Thm: 3} could be replaced by $1<p$ since if $p=1$ then $(p,q)=1\mid (b-a) $ for all $a,b\in\mathbb{Z}$. Note that in cases (2) and (3) of Theorem \ref{Thm: 3}, we have the same Minkowski dimensions, but the Hausdorff dimensions are different.

The Hausdorff dimension result for $X_{A}^{(q)}$ (Theorem \ref{Thm: 2} (2)) is extended to affine multiplicative SFT $X_{A}^{(p,q;a,b)}$ when comparing formula (\ref{6}) with (\ref{8}). However, we point out that studying the dimensions of $X_{A}^{(p,q;a,b)}$ is quite different to that of $X_{A}^{(q)}$ for two reasons.

\textbf{1}. The decomposition of the lattice $\mathbb{N}$ into the the correlated sublattices according to the `affine multiplicative constraint' $(pk+a,qk+b)$ is totally different from the case of the `multiplicative constraint' $\left(k, qk\right) $. For example, for the case where $(p,q;a,b)=(1,q;0,0)$ (cf. \cite{ban2021entropy,fan2014some, kenyon2012hausdorff}), the lattice $\mathbb{N}$ can be decomposed according to the correlated sublattices as \begin{equation}
\mathbb{N}=\bigsqcup\limits_{q\nmid i}\{i,iq,iq^{2},\ldots \}\text{,}
\label{4}
\end{equation}
where each $\{i,iq,iq^{2},\ldots \}$ is an infinite string of lattices in $\mathbb{N}$. However, obtaining a closed form for the decomposition of $ \mathbb{N}$ as shown in (\ref{4}) is extremely challenging for general $(p,q;a,b)$. This is because, for $n=qk+b\in\mathbb{N}$, $n$ may not necessarily be of the form $pk'+a$. Thus, unlike the case $(1,q;0,0)$, the decomposition of $\mathbb{N}$ for the general $(p,q;a,b)$ may contain arbitrary long strings of lattices in $\mathbb{N}$.  

\textbf{2}. Consider the issue we mentioned in the preceding paragraph. The classical optimization method in \cite{kenyon2012hausdorff} for finding the Hausdorff dimension is not valid for general $(p,q;a,b)$.

We offer some new ideas to overcome the aforementioned difficulties and derive the Hausdorff and Minkowski dimensions for the affine multiplicative SFTs. Section \ref{sec2} yields the Minkowski dimension and Section \ref{sec3} gives the Hausdorff dimension for $X_{A}^{(p,q;a,b)}$. In Section \ref{sec4}, we discuss the
dimension theory for higher-order multiplicative constraints.

\section{Proof of Theorem \ref{Thm: 1}}\label{sec2}
In this section, we provide the proof of Theorem \ref{Thm: 1}.  
\begin{proof}[Proof of Theorem \ref{Thm: 1}]
	\item[\bf (1)] For $(p,q)\nmid (b-a)$, we observe that the equation
	\[
	qk+b=pk'+a
	\]
	has no integer pair $(k,k')$ as a solution. This observation gives that the constraints $(pk+a,qk+b)$ only induce two types of disjoint lattices with cardinalities 1 and 2. 
 
 Now, we compute the number of each type of lattices appearing in a finite interval. For $n\geq 1$, let $\mathcal{A}(p,q,a,b,n)$ be the collection of disjoint lattices with cardinality 2 intersecting $\{1,...,n\}$ with respect to the constraints $(pk+a,qk+b)$. Then,
\begin{align*}
    \#\mathcal{A}(p,q,a,b,n)=&\#\left\{k\geq 1: 1\leq pk+a,qk+b\leq n\right\}\\
    =&\#\left\{ k\geq 1: \frac{1-b}{q}\leq k \leq \frac{n-b}{q}\mbox{ and }\frac{1-a}{p}\leq k \leq \frac{n-a}{p}\right\}.
\end{align*}
	Since $1\leq p<q$ and $a,b$ are fixed, we have 
 \begin{equation*}
    \frac{n-b}{q}\leq \frac{n-a}{p}\mbox{ as }n\gg 1.
 \end{equation*}
Thus, if $n\gg 1$,
\begin{align*}
    \#\mathcal{A}(p,q,a,b,n)=\#\left\{ k: \max\left\{1,\frac{1-a}{p},\frac{1-b}{q}\right\}\leq k \leq \frac{n-b}{q}\right\},
\end{align*}
and
	\begin{equation}\label{eq 2.1-1}
	\lim_{n\to \infty}\frac{\# \mathcal{A}(p,q,a,b,n)}{n}=\frac{1}{q}.
	\end{equation}
	
	Observe that the members in 
 \begin{equation*}
     \left\{1,...,n\right\}\setminus \mathcal{A}(p,q,a,b,n)
 \end{equation*}
 are disjoint lattices with cardinality 1 intersecting $\{1,...,n\}$ with respect to the constraints $(pk+a,qk+b)$. Note that 
	\begin{equation}\label{eq 2.1-2}
	  \lim_{n\to\infty} \frac{n-2	\# \mathcal{A}(p,q,a,b,n)}{n}=1-\frac{2}{q}.  
	\end{equation}
Then, 
\begin{equation}\label{eq 2.1-3}
   \left|\mathcal{P}\left(X^{p,q;a,b}_{A},\{1,...,n\}\right)\right|=m^{n-2	\# \mathcal{A}(p,q,a,b,n)}\left|A\right|^{\# \mathcal{A}(p,q,a,b,n)}, 
\end{equation}
where $\mathcal{P}(X,I):=\left\{x|_I: x\in X \right\}$.

Therefore, by (\ref{eq 2.1-1}), (\ref{eq 2.1-2}) and (\ref{eq 2.1-3}), 
\begin{align*}
\dim_M X^{p,q;a,b}_{A}&=\lim_{n\to\infty}\frac{\log_m\left|\mathcal{P}\left(X^{p,q;a,b}_{A},\{1,...,n\}\right)\right|}{n}\\
    &=1-\frac{2}{q} +\frac{1}{q}\log_m |A|.
\end{align*}
	
\item[\bf (2)] For $(p,q)\mid (b-a)$, the equation
	\[
	qk+b=pk'+a
	\]
	has an integer pair $(k,k')$ as a solution. The following observation gives that the constraints $(pk+a,qk+b)$ induce to many types of disjoint lattices, i.e. the lattice of cardinality $\ell$ for all $\ell\in\mathbb{N}$. 
 
 For $n\geq 1$, we observe that the disjoint lattices with cardinality 2 intersecting $\{1,...,n\}$ with respect to the constraints $(pk+a,qk+b)$ are of the form $\left\{A_1,A_2\right\}$, where
 \begin{equation*}
     A_1=pk+a \mbox{ and }A_2=qk+b.
 \end{equation*}
Let $p_1=\frac{p}{(p,q)}$, $q_1=\frac{q}{(p,q)}$ and $c=\frac{b-a}{(p,q)}$. Since $(p,q)\mid (b-a)$ and $(p_1,q_1)=1$, there exists a unique $0\leq r_1\leq p_1-1$ such that
 \[
 \frac{q_1r_1+c}{p_1}\in\mathbb{Z}.
 \]
 This gives that the disjoint lattices with cardinality 3 intersecting $\{1,...,n\}$ with respect to the constraints $(pk+a,qk+b)$ are of the form $\{A_1,A_2,A_3\}$, where
\begin{align*}
    A_1&=p\left(p_1k_1+r_1\right)+a,\\
    A_2&=q\left(p_1k_1+r_1\right)+b=p\left(q_1k_1+\frac{q_1r_1+c}{p_1}\right)+a,\\
    A_3&=q\left(q_1k_1+\frac{q_1r_1+c}{p_1}\right)+b.
\end{align*}
Note that 
\[
A_3=(p,q)q_1^2\left(k_1+\frac{r_1}{p_1}+\frac{c}{p_1q_1}\right)+b.
\]
There exists a unique $0\leq r_2\leq p_1-1$ such that the last term of disjoint lattices with cardinality 4 intersecting $\{1,...,n\}$ with respect to the constraints $(pk+a,qk+b)$ are of the form
	\begin{align*}
	&q_1\frac{(p,q)q_1^2\left(\left(p_1k_2+r_2\right)+\frac{r_1}{p_1}+\frac{c}{p_1q_1}+\frac{c}{q_1^2}\right)}{p_1}+b\\
 =&(p,q)q_1^3\left(k_2+\frac{r_2}{p_1}+\frac{r_1}{p_1^2}+\frac{c}{p_1^2q_1}+\frac{c}{p_1q_1^2}\right)+b.
	\end{align*}

Then, by a similar process, for $\ell\geq 3$, there exists a sequence $\{r_i\}_{i=1}^{\ell-2}$ with $0\leq r_i\leq p_1-1$ such that the last term of disjoint lattices with cardinality $\ell$ intersecting $\{1,...,n\}$ with respect to the constraints $(pk+a,qk+b)$ are of the form 
\begin{equation}\label{eq 2.1-4}
    (p,q)q_1^{\ell-1}\left(k_{\ell-2}+a_{\ell}+b_{\ell}\right)+b,
\end{equation}
where 
\[a_\ell=\sum_{i=1}^{\ell-2}\frac{r_{\ell-1-i}}{p_1^i}\quad \mbox{and}\quad b_\ell=\sum_{i=1}^{\ell-2}\frac{c}{p_1^i q_1^{\ell-1-i}}\quad (\forall \ell\geq 2).\]
 
Now, we compute the number of disjoint lattices having $\ell$ ($\ell\geq 2$) members of $\{1,...,n\}$. By (\ref{eq 2.1-4}), we have that this number is equal to the number of $k$ satisfying  
\[
\left\lfloor\frac{n-b}{(p,q)q_1^\ell}-a_{\ell+1}-b_{\ell+1}\right\rfloor< k\leq \left\lfloor\frac{n-b}{(p,q)q_1^{\ell-1}}-a_{\ell}-b_{\ell}\right\rfloor.
\]

In fact, the probability that a lattice has its smallest member not of the form $qk+b$ is $\frac{q_1-1}{q_1}$. Thus, the number of disjoint lattices having exactly $\ell$ ($\ell\geq 1$) members of $\{1,...,n\}$ is 
\begin{equation}\label{eq 2.1-5}
   D_\ell(n):= \frac{q_1-1}{q_1}\left(\left\lfloor\frac{n-b}{(p,q)q_1^{\ell-1}}-a_\ell-b_\ell\right\rfloor-\left\lfloor\frac{n-b}{(p,q)q_1^\ell}-a_{\ell+1}-b_{\ell+1}\right\rfloor\right),
\end{equation}

For all $i\geq 1$, $0\leq r_i \leq p_1-1$, we have 
\begin{equation}\label{eq 2.1-6}
    a_\ell\leq \sum_{i=1}^\infty \frac{p_1-1}{p_1^i}<\infty~ (\forall \ell \geq 2),
\end{equation}
and since $1\leq p_1<q_1$, we have 
\begin{equation}
    b_\ell\leq \frac{(\ell-2)c}{p_1^{\ell-1}} <\infty~(\forall \ell \geq 2).
\end{equation}
 Thus, by (\ref{eq 2.1-4}), (\ref{eq 2.1-5}) and (\ref{eq 2.1-6}), the density of disjoint lattices that intersect $\{1,...,n\}$ exactly $\ell\geq 2$ members (as $n\to\infty$) is 
\begin{equation}\label{eq 2.1-7}
\lim_{n\to\infty}\frac{D_\ell(n)}{n}=\frac{(q_1-1)^2}{(p,q)q_1^{\ell+1}}.
\end{equation}

It remains to compute the density of disjoint lattices that intersect $\{1,...,n\}$ exactly one member as $n\to\infty$. By (\ref{eq 2.1-7}), the density is 
\begin{equation}\label{eq 2.1-8}
  \lim_{n\to\infty}\frac{n-\sum_{\ell=2}^n D_\ell(n)}{n}= 1-\sum_{\ell=2}^\infty\frac{\ell(q_1-1)^2}{(p,q)q_1^{\ell+1}}=1-\frac{2q_1-1}{(p,q)(q_1)^2}. 
\end{equation}

On the other hand, for $n\geq 1$, we have
\begin{equation}\label{eq 2.1-9}
    \left|\mathcal{P}\left(X^{p,q;a,b}_{A},\{1,...,n\}\right)\right|=m^{n-\sum_{\ell=2}^n D_\ell(n)}\prod_{\ell=2}^{n}\left|A^{\ell-1}\right|^{D_\ell(n)}.
\end{equation}
Therefore, by (\ref{eq 2.1-7}), (\ref{eq 2.1-8}) and (\ref{eq 2.1-9}), we have 
\begin{align*}
\dim_MX^{p,q;a,b}_{A}&=\lim_{n\to\infty}\frac{\log_m\left|\mathcal{P}\left(X^{p,q;a,b}_{A},\{1,...,n\}\right)\right|}{n}\\
    &=1-\frac{2q_1-1}{(p,q)(q_1)^2}+\frac{(q_1-1)^2}{(p,q)}\sum_{i=2}^\infty \frac{\log_m |A^{i-1}|}{(q_1)^{i+1}}.
\end{align*}
\end{proof}

\section{Proof of Theorem \ref{Thm: 3}}\label{sec3}
Let $X_A\subseteq \Sigma_m^{\mathbb{N}}$ be the SFT associated with a transition matrix $A$. For $k\geq 1$, let $\alpha_k$ be the family of $k$-cylinders of $X_A$, that is,
\begin{equation*}
\alpha_k=\left\{[x_1,...,x_k]: x=(x_i)_{i=1}^\infty \in X_A\right\},
\end{equation*}
where 
\begin{equation*}
    [x_1,...,x_n]=\left\{y=(y_i)_{i=1}^\infty\in X_A: y_i=x_i \mbox{ for all }1\leq i\leq n \right\}.
\end{equation*}
Let $\mu$ be a probability measure on $X_A$, the $\mu$-entropy $H^\mu (\alpha_k)$ of $\alpha_k$ is defined by 
\begin{align*}
H^\mu (\alpha_k)=-\sum_{B\in \alpha_k} \mu(B)\log_m \mu(B).
\end{align*}

The following lemma is needed for the lower and upper bounds of the Hausdorff dimension of $X_A^{p,q;a,b}$.
\begin{lemma}[Billingsley's lemma]\label{lemma B}
    Let $E$ be a Borel set in $\Sigma_m^{\mathbb{N}}$ and let $\nu$ be a finite Borel measure on $\Sigma_m^{\mathbb{N}}$.
    \begin{enumerate}
    \item We have $\dim_HE\geq c$ if $\nu(E)>0$ and 
    \begin{equation*}
        \liminf_{n\to\infty}\frac{-\log_m \nu([x_1x_2\cdots x_n])}{n}\geq c \quad\mbox{for }\nu\mbox{-a.e. }x\in E.
    \end{equation*}
    \item We have $\dim_HE\leq c$ if 
    \begin{equation*}
        \liminf_{n\to\infty}\frac{-\log_m \nu([x_1x_2\cdots x_n])}{n}\leq c \quad\mbox{for all }x\in E.
      \end{equation*}
      \end{enumerate}
\end{lemma}

The following proposition gives the lower bound of the Hausdorff dimension of $X_A^{p,q;a,b}$.
\begin{proposition}[Lower bound]\label{lemma 3.1}
Let $A=[a_{i,j}]_{m\times m}$ be an $m\times m$ irreducible binary matrix, $1\leq p<q$ and $a,b\in\mathbb{Z}$.
\begin{enumerate}
\item If $(p,q)\nmid (b-a)$, then 
\begin{align*}
\dim_HX^{p,q;a,b}_A&\geq 1-\frac{1}{p}-\frac{1}{q}+\frac{1}{p}\log_m \sum_{i=0}^{m-1}\left(a_i\right)^{\frac{p}{q}},
\end{align*}
where $a_i=\sum_{j=0}^{m-1} a_{i,j}$.
\item If $(p,q)\mid (b-a)$ and $p>1$, then 
\begin{align*}
	\dim_HX^{p,q;a,b}_A\geq 1-\frac{1}{p}-\frac{1}{q}+\frac{1}{\left(p,q\right)p_1q_1}+\sum_{i=2}^\infty\left(\sum_{j=1}^i P_{i,j}\right) \log_m t_{\phi;i},
\end{align*}	
where $P_{i,j}$ and $t_{\phi;i}$ are defined in (\ref{eq 3.1-1}) and (\ref{eq 3.1-2}).
\item If $p=1$ and $A$ is primitive, then 
\[
\dim_H X^{p,q;a,b}_{A}\geq\frac{q-1}{q} \log_m \sum_{i=0}^{m-1} t_i ,
\]
where $(t_i)_{i=0}^{m-1}$ is a unique positive vector satisfying $t_i^q=\sum_{j=0}^{m-1} a_{i,j}t_j$.
\end{enumerate}
\end{proposition}

\begin{proof}
    \item[\bf (1)] The idea of the proof is as follows. \\
    (i) Decompose $\mathbb{N}=A_1 \sqcup A_2$ as in Theorem \ref{Thm: 1}, where
    $A_1$ is the collection of disjoint lattices with cardinality 1, and $A_2$ is the collection of disjoint lattices with cardinality 2.\\
    (ii) Let $\mu_1$ be a probability measure on a lattice with cardinality 1, and $\mu_2$ be a probability measure on the lattice of the form $\{pk+a,qk+b\}$.\\
    (iii) Define a probability measure on $X_A^{p,q;a,b}$ by defining the measure on cylinder sets in $\alpha_n$ of $X_A^{p,q;a,b}$. That is, for any $[x_1,...,x_n]\in \alpha_n$,
    \begin{align*}
        \mathbb{P}_{\mu_1,\mu_2}([x_1\cdots x_n])&:=\prod_{U\in A_1\cap [1,n]}\mu_1([x_{U}])\prod_{V\in A_2\cap [1,n]}\mu_2([x_{V}]),
    \end{align*}
    where $A_i\cap [1,n]=\{U\cap [1,n]: U\in A_i\}$.\\
    (iv) Compute the following limits:
    \begin{equation}\label{eq L1}
    \begin{aligned}
        \lim_{n\to\infty}\frac{L(1,n,1)}{n}&=1-\frac{1}{p}-\frac{1}{q},\\
        \lim_{n\to\infty}\frac{L(2,n,1)}{n}&=\frac{1}{p}-\frac{1}{q},\\
        \lim_{n\to\infty}\frac{L(2,n,2)}{n}&=\frac{1}{q},
    \end{aligned}
    \end{equation}
    where 
    \begin{equation*}
        L(i,n,j)=\#\{U:U\in A_i\cap [1,n],|U|=j\}.
    \end{equation*}
    (v) Maximize the following two quantities:
\begin{align*}
    & \max_{\mu_1} \left(1-\frac{1}{p}-\frac{1}{q}\right)H^{\mu_1}(\alpha_1),\\
     &\max_{\mu_2} \left(\frac{1}{p}-\frac{1}{q}\right) H^{\mu_2}(\alpha_1)+ \frac{1}{q}H^{\mu_2}(\alpha_2).
\end{align*}
    Then, 
    \begin{equation}\label{eq 3.1-3}
    \begin{aligned}
        &\mu_1([i])=\frac{1}{m}~(\forall 0\leq i \leq m-1),\\
        &\mu_2([i])=\frac{(a_i)^{\frac{p}{q}}}{\sum_{j=0}^{m-1}\left(a_j\right)^{\frac{p}{q}}}~(\forall 0\leq i\leq m-1),\\
        &\mu_2([ij])=\frac{(a_i)^{\frac{p}{q}-1}a_{i,j}}{\sum_{j=0}^{m-1}\left(a_j\right)^{\frac{p}{q}}}~(\forall 0\leq i,j \leq m-1),
    \end{aligned}
    \end{equation}
where $a_i=\sum_{j=0}^{m-1} a_{i,j}$ for all $0\leq i \leq m-1$.
    
 By (\ref{eq L1}), (\ref{eq 3.1-3}) and a similar process in \cite[Equation (28)]{kenyon2012hausdorff}, we have that for $ \mathbb{P}_{\mu_1,\mu_2}$-a.e. $x\in X^{p,q;a,b}_{A}$, 
\begin{align*}
&\liminf_{n\to\infty}\frac{-\log_m \mathbb{P}_{\mu_1,\mu_2}([x_1x_2\cdots x_n])}{n}\\
\geq &\left(1-\frac{1}{p}-\frac{1}{q}\right)H^{\mu_1}\left(\alpha_1\right)+\left(\frac{1}{p}-\frac{1}{q}\right)H^{\mu_2}\left(\alpha_1\right)+\frac{1}{q}H^{\mu_2}\left(\alpha_2\right)\\
=&\left(1-\frac{1}{p}-\frac{1}{q}\right)\sum_{i=0}^{m-1}-\mu_1([i])\log_m \mu_1([i])+\left(\frac{1}{p}-\frac{1}{q}\right)\sum_{i=0}^{m-1}-\mu_2([i])\log_m \mu_2([i])\\
&+\frac{1}{q}\sum_{i,j=0}^{m-1}-\mu_2([ij])\log_m \mu_2([ij])\\
=&1-\frac{1}{p}-\frac{1}{q}+\frac{1}{p}\log_m \sum_{i=0}^{m-1}\left(a_i\right)^{\frac{p}{q}}.
\end{align*}
By Lemma \ref{lemma B}, 
\begin{align*}
\dim_HX^{p,q;a,b}_A&\geq 1-\frac{1}{p}-\frac{1}{q}+\frac{1}{p}\log_m \sum_{i=0}^{m-1}\left(a_i\right)^{\frac{p}{q}}.
\end{align*}

\item[\bf (2)] A similar idea from previous case is applicable to the second case.\\
    (i) Decompose $\mathbb{N}=\sqcup_{i=1}^\infty A_i $ as in Theorem \ref{Thm: 1}, where $A_i$ is the collection of disjoint lattices with cardinality $i$.\\
    (ii) Let $\mu_i$ be a probability measure on a lattice of the form 
    \[\{k,g(k),g(g(k)),...,g^{i-1}(k)\},\]
    where $k-b\neq 0$ (mod $q_1$), $g^j(k)-a= 0$ (mod $p_1$) for all $0\leq j \leq i-1$ and $g^i(k)-a\neq 0$ (mod $p_1$) with $g(x):=q\frac{x-a}{p}+b$.\\
    (iii) Define a probability measure on $X_A^{p,q;a,b}$ by defining the measure on cylinder sets in $\alpha_n$ of $X_A^{p,q;a,b}$. That is, for any $[x_1,...,x_n]\in \alpha_n$,
    \begin{align*}
        \mathbb{P}_{\infty}([x_1\cdots x_n])&:=\prod_{i=1}^\infty\prod_{U\in A_i\cap [1,n]}\mu_i([x_{U}]).
    \end{align*}
    (iv) Compute the following limits:
      \begin{equation}\label{Pij}
        P_{i,j}=\lim_{n\to\infty}\frac{L(i,n,j)}{n}.        
 \end{equation}
For each disjoint lattice with cardinality $i$ in $[1,n]$, it has a probability of $\frac{1}{p_1}$ in $\cup_{j\geq  i+1}A_j$ and a probability of $\frac{p_1-1}{p_1}$ in $A_i$. Similarly, it has a probability of $(\frac{1}{p_1})^{k+1}$ in $\cup_{j\geq  i+k+1}A_j$ and a probability of $(\frac{1}{p_1})^k\frac{p_1-1}{p_1}$ in $A_{i+k}$ for $k\geq 0$. Thus, we have
 \begin{equation}\label{eq 3.1-1}
     \begin{aligned}
     P_{1,1}&=1-\frac{1}{p}-\frac{1}{q}+\frac{1}{(p,q)p_1q_1},\\
     P_{i,1}&=\frac{q_1-1}{q_1}\frac{p_1-1}{p_1}\left(\frac{1}{p}-\frac{1}{q}\right)\left(\frac{1}{p_1}\right)^{i-2},\\
     P_{i,j}&=\frac{\left(q_1-1\right)^2}{\left(p,q\right)\left(q_1\right)^{j+1}} \frac{p_1-1}{\left(p_1\right)^{i-j+1}}~(\forall i\geq j\geq 2).
\end{aligned}
 \end{equation}
(v) Maximize the following quantities:
\begin{equation*}
   \max_{\mu_i} P_{i,1} H^{\mu_i}(\alpha_1)+\sum_{j=2}^i P_{i,j}H^{\mu_{i}}(\alpha_j),~ (\forall i\geq 2).
\end{equation*}
 For $i=2$, we obtain 
\begin{align*}
    &\mu_2([i_1])=\frac{(a_{i_1})^{\frac{P_{2,2}}{P_{2,1}+P_{2,2}}}}{t_{\phi;2}}~(\forall 0\leq i_1 \leq m-1),\\
    &\mu_2([i_1i_2])=\frac{a_{i_1,i_2}(a_{i_1})^{\frac{-P_{2,1}}{P_{2,1}+P_{2,2}}}}{t_{\phi;2}} ~(\forall 0\leq i_1,i_2 \leq m-1),
\end{align*}
where
\begin{equation*}
    t_{\phi;2}=\sum_{j=0}^{m-1}\left(a_j\right)^{\frac{P_{2,2}}{P_{2,1}+P_{2,2}}}.
\end{equation*}
Then,
\begin{equation*}
   P_{2,1} H^{\mu_2}(\alpha_1)+P_{2,2} H^{\mu_2}(\alpha_2)=\left(P_{2,1}+P_{2,2}\right)\log_m t_{\phi;2}.
\end{equation*}

For $i=3$ and $0\leq i_1,i_2,i_3\leq m-1$, we have
\begin{align*}
    &\mu_3([i_1i_2i_3])
    =\frac{a_{i_1,i_2}a_{i_2,i_3}(a_{i_2})^{\frac{-P_{3,2}}{P_{3,2}+P_{3,3}}}\left(\sum_{j=0}^{m-1}a_{i_1,j}(a_j)^{\frac{P_{3,3}}{P_{3,2}+P_{3,3}}}\right)^{\frac{-P_{3,1}}{\sum_{k=1}^3 P_{3,k}}}}{t_{\phi;3}},\\
    &\mu_3([i_1i_2])
    =\frac{a_{i_1,i_2}(a_{i_2})^{\frac{P_{3,3}}{P_{3,2}+P_{3,3}}}\left(\sum_{j=0}^{m-1}a_{i_1,j}(a_j)^{\frac{P_{3,3}}{P_{3,2}+P_{3,3}}}\right)^{\frac{-P_{3,1}}{\sum_{k=1}^3 P_{3,k}}}}{t_{\phi;3}},\\
    &\mu_3([i_1])
    =\frac{\sum_{i_2=0}^{m-1}a_{i_1,i_2}(a_{i_2})^{\frac{P_{3,3}}{P_{3,2}+P_{3,3}}}\left(\sum_{j=0}^{m-1}a_{i_1,j}(a_j)^{\frac{P_{3,3}}{P_{3,2}+P_{3,3}}}\right)^{\frac{-P_{3,1}}{\sum_{k=1}^3 P_{3,k}}}}{t_{\phi;3}},
\end{align*}
where
\begin{equation*}
    t_{\phi;3}=\sum_{r_1=0}^{m-1}\sum_{r_2=0}^{m-1}a_{r_1,r_2}(a_{r_2})^{\frac{P_{3,3}}{P_{3,2}+P_{3,3}}}\left(\sum_{j=0}^{m-1}a_{r_1,j}(a_j)^{\frac{P_{3,3}}{P_{3,2}+P_{3,3}}}\right)^{\frac{-P_{3,1}}{\sum_{k=1}^3 P_{3,k}}}.
\end{equation*}
Then,
\begin{align*}
    P_{3,1} H^{\mu_3}(\alpha_1)+P_{3,2} H^{\mu_3}(\alpha_2)+P_{3,3}H^{\mu_3}(\alpha_3)
    =\left(\sum_{k=1}^3 P_{3,k}\right)\log_m t_{\phi;3}.
\end{align*}

Similarly, for $i\geq 2$,
\begin{align*}
    &\mu_{i}([j_1j_2\cdots j_i])=a_{j_1,j_2}\cdots a_{j_{i-1},j_i}\frac{t_{j_1,...,j_i;i}}{t_{\phi;i}},
\end{align*}
where
\begin{align*}
    t_{j_1,...,j_i;i}&=\prod_{k=1}^{i-1}f_k(j_{i-k},i),
\end{align*}
and for $1\leq k \leq i-1$, $f_{k}(j_{i-k},i)$ is defined recursively as
\begin{align*}
f_k(j_{i-k},i)&=\left[\sum_{b_1,b_2,...,b_{k-1}=0}^{m-1}a_{j_{i-k},b_1}\prod_{d=1}^{k-2}a_{b_d,b_{d+1}}a_{b_{k-1}}\prod_{c=1}^{k-1}f_{c}(b_{k-c},i)\right]^{\frac{-P_{i,i-k}}{\sum_{\ell=0}^k P_{i,i-\ell}}}
\end{align*}
 and
\begin{equation}\label{eq 3.1-2}
    t_{\phi;i}=\sum_{j_1,...,j_i=0}^{m-1}a_{j_1,j_2}a_{j_2,j_3}\cdots a_{j_{i-1},j_i}t_{j_1,...,j_i;i}.
\end{equation}

For $1\leq k \leq i-1$,
\begin{equation*}
    \mu_{i}([j_1j_2\cdots j_k])=\sum_{\ell=0}^{m-1}\mu_{i}([j_1j_2\cdots j_k \ell]).
\end{equation*}
Thus,
\begin{equation}\label{eq 3.1-5}
    P_{i,1} H^{\mu_i}(\alpha_1)+\sum_{j=2}^i P_{i,j}H^{\mu_{i}}(\alpha_j)=\left(\sum_{j=1}^i P_{i,j}\right)\log_m t_{\phi;i}.
\end{equation}

Note that $\mu_1([i_1])=\frac{1}{m}$ for all $0\leq i_1\leq m-1$. Combining (\ref{eq 3.1-1}), (\ref{eq 3.1-5}) and a similar process in \cite[Equation (28)]{kenyon2012hausdorff}, we have that for $ \mathbb{P}_{\infty}$-a.e. $x\in X^{p,q;a,b}_{A}$,
\begin{align*}
&\liminf_{n\to\infty}\frac{-\log_m \mathbb{P}_{\infty}([x_1x_2\cdots x_n])}{n}\\
\geq &\left(1-\frac{1}{p}-\frac{1}{q}+\frac{1}{(p,q)p_1q_1}\right)H^{\mu_1}\left(\alpha_1\right)+\sum_{i=2}^\infty\left[  P_{i,1} H^{\mu_i}(\alpha_1)+\sum_{j=2}^i P_{i,j}H^{\mu_{i}}(\alpha_j)\right]\\
=&1-\frac{1}{p}-\frac{1}{q}+\frac{1}{(p,q)p_1q_1}+\sum_{i=2}^\infty\left(\sum_{j=1}^i P_{i,j}\right) \log_m t_{\phi;i}.
\end{align*}
By Lemma \ref{lemma B},
\begin{align*}
	\dim_HX^{p,q;a,b}_A\geq& 1-\frac{1}{p}-\frac{1}{q}+\frac{1}{\left(p,q\right)p_1q_1}+\sum_{i=2}^\infty\left(\sum_{j=1}^i P_{i,j}\right) \log_m t_{\phi;i}.
\end{align*}	
    \item[\bf (3)] The proof is similar to that of Theorem \ref{Thm: 2}, we omit it here. 
\end{proof}

Now, we are almost ready for proving Theorem \ref{Thm: 3}. We need the following H\"{o}lder Inequality \cite[Section 2.7]{inequality1952}.

\begin{lemma}[H\"{o}lder Inequality]\label{lemma H}
    Let $\lambda_a$ and $\lambda_b$ be positive real numbers with $\lambda_a+\lambda_b=1$. Let $a_1,...,a_n,b_1,...,b_n$ be positive real numbers. Then
\begin{equation*}
    (a_1+\cdots+a_n)^{\lambda_a} (b_1+\cdots +b_n)^{\lambda_b}\geq \sum_{i=1}^n a_i^{\lambda_a}b_i^{\lambda_b}.
\end{equation*}
Equality holds if $a_1:a_2:\cdots:a_n\equiv b_1:b_2:\cdots :b_n $.
\end{lemma}

\begin{proof}[Proof of Theorem \ref{Thm: 3}] 
\item[\bf (1)] Fix any $x=(x_k)_{k=1}^\infty\in X^{p,q;a,b}_A$, we claim that
\begin{equation}\label{claim1}
     \liminf_{n\to\infty}\frac{-\log_m \mathbb{P}_{\mu_1,\mu_2}([x_1x_2\cdots x_n])}{n}\leq 1-\frac{1}{p}-\frac{1}{q}+\frac{1}{p}\log_m \sum_{i=0}^{m-1}\left(a_i\right)^{\frac{p}{q}}.
\end{equation}

In fact, by (\ref{eq 3.1-3}), we have 
\begin{align*}
    &-\log_m \mathbb{P}_{\mu_1,\mu_2}([x_1x_2\cdots x_n])\\
    =&\sum_{U\in A_1\cap [1,n]}-\log_m \mu_1([x_U])+\sum_{V\in A_2\cap [1,n],|V|=1}-\log_m\mu_2([x_V])\\
    &+\sum_{V\in A_2\cap [1,n],|V|=2}-\log_m\mu_2([x_V])\\
    =&\sum_{U\in A_1\cap [1,n]}-\log_m \frac{1}{m}+\sum_{V\in A_2\cap [1,n],|V|=1}-\log_m\frac{\nu_2([x_V])}{\sum_{j=0}^{m-1}(a_j)^{\frac{p}{q}}}\\
    &+\sum_{V\in A_2\cap [1,n],|V|=2}-\log_m\frac{\nu_2([x_V])}{\sum_{j=0}^{m-1}(a_j)^{\frac{p}{q}}}\\
    =&\left(1-\frac{1}{p}-\frac{1}{q}\right)n\\
    &+\left(\frac{1}{p}-\frac{1}{q}\right)n\log_m \sum_{j=0}^{m-1}(a_j)^{\frac{p}{q}}+\sum_{V\in A_2\cap [1,n],|V|=1}-\log_m\nu_2([x_V])\\
    &+\frac{n}{q}\log_m \sum_{j=0}^{m-1}(a_j)^{\frac{p}{q}}+\sum_{V\in A_2\cap [1,n],|V|=2}-\log_m\nu_2([x_V]),
\end{align*}
where $\nu_2([i])=(a_i)^{\frac{p}{q}}$ and $\nu_2([ij])=(a_i)^{\frac{p}{q}-1}a_{i,j}$.

Since
\begin{align*}
    &\sum_{V\in A_2\cap [1,n],|V|=1}-\log_m\nu_2([x_V])+\sum_{V\in A_2\cap [1,n],|V|=2}-\log_m\nu_2([x_V])\\
    =&\sum_{i=0}^{m-1}L(2,n,1,i,x)\left(-\log_m\nu_2([i])\right)+\sum_{i,j=0}^{m-1}L(2,n,2,ij,x)\left(-\log_m\nu_2([ij])\right)\\
    =&\sum_{i=0}^{m-1}L(2,n,1,i,x)\frac{-p}{q}\log_m a_i+\sum_{i,j=0}^{m-1}L(2,n,2,ij,x)\left(1-\frac{p}{q}\right)\log_m a_i,
\end{align*}
then, by (\ref{eq L1}), we have
\begin{align*}
    &\frac{1}{n}\left(\sum_{V\in A_2\cap [1,n],|V|=1}-\log_m\nu_2([x_V])+\sum_{V\in A_2\cap [1,n],|V|=2}-\log_m\nu_2([x_V])\right)\\
    =&\left[\left(\frac{1}{p}-\frac{1}{q}\right)\frac{\sum_{i=0}^{m-1}L(2,n,1,i,x)}{L(2,n,1)}\frac{-p}{q}+\frac{1}{q}\frac{\sum_{i,j=0}^{m-1}L(2,n,2,ij,x)}{L(2,n,2)}\left(1-\frac{p}{q}\right)\right]\log_m a_i,
\end{align*}
where 
\begin{equation*}
    L(k,n,s,v_1\cdots v_s,x)=\#\{U:U\in A_k\cap [1,n],|U|=s,x_U=v_1\cdots v_s\}.
\end{equation*}
Without loss of generality, after choosing further subsequence finitely many times, we may assume that there exists a subsequence $\{\ell_k\}_{k=1}^\infty$ of $\mathbb{N}$ such that the following limits exist:
\begin{align*}
   R_{v_1}(x)&=\lim_{k\to\infty} \frac{L(2,\ell_k,1,v_1,x)}{L(2,\ell_k,1)}~(\forall 0\leq v_1 \leq m-1),\\
   R_{v_1v_2}(x)&=\lim_{k\to\infty} \frac{L(2,\ell_k,2,v_1v_2,x)}{L(2,\ell_k,2)} ~ (\forall 0\leq v_1,v_2 \leq m-1).
\end{align*}
Then, it is not difficult to see that
\begin{align*}
   R_{v_1}'(x)&=R_{v_1}(x) ~(\forall 0\leq v_1 \leq m-1),
\end{align*}
where $R_{v_1}'(x):= \sum_{v_2=0}^{m-1}R_{v_1 v_2}(x)$. 

Then,
	\begin{align*}
	 &\liminf_{n\to\infty}\frac{-\log_m \mathbb{P}_{\mu_1,\mu_2}([x_1x_2\cdots x_n])}{n}\\
  \leq &1-\frac{1}{p}-\frac{1}{q}+\frac{1}{p}\log_m \sum_{i=0}^{m-1}\left(a_i\right)^{\frac{p}{q}}+\left(\frac{1}{q}-\frac{p}{q^2}\right)\sum_{i=0}^{m-1}\left|R_{i}'(x)-R_i(x)\right|\log_m a_i.
\end{align*}
Since $R_i'(x)=R_i(x)$ for all $0\leq i \leq m-1$ and for all $x\in X^{p,q;a,b}_A$, we derive our claim (\ref{claim1}).

Then, by Lemma \ref{lemma B},
	\begin{equation*}
	\dim_HX^{p,q;a,b}_A\leq 1-\frac{1}{p}-\frac{1}{q}+\frac{1}{p}\log_m \sum_{i=0}^{m-1}\left(a_i\right)^{\frac{p}{q}}.
	\end{equation*}
Finally, by Proposition \ref{lemma 3.1}, we have
 \begin{equation*}
	\dim_HX^{p,q;a,b}_A= 1-\frac{1}{p}-\frac{1}{q}+\frac{1}{p}\log_m \sum_{i=0}^{m-1}\left(a_i\right)^{\frac{p}{q}}.
	\end{equation*}

\item[\bf (2)] Fix any $x=(x_k)_{k=1}^\infty\in X^{p,q;a,b}_A$. We claim that
\begin{equation}\label{claim2}
    \begin{aligned}
     &\liminf_{n\to\infty}\frac{-\log_m \mathbb{P}_{\infty}([x_1x_2\cdots x_n])}{n}\\
     \leq& 1-\frac{1}{p}-\frac{1}{q}+\frac{1}{\left(p,q\right)p_1q_1}+\sum_{i=2}^\infty\left(\sum_{j=1}^i P_{i,j}\right) \log_m t_{\phi;i}.
\end{aligned}
\end{equation}

In fact, we have
\begin{align*}
    &-\log_m \mathbb{P}_{\infty}([x_1x_2\cdots x_n])\\
    =&\sum_{i=1}^\infty\sum_{j=1}^i \sum_{U\in A_i\cap [1,n],|U|=j}-\log_m \mu_i([x_U])\\
    =&\sum_{U\in A_1\cap [1,n]}-\log_m \frac{1}{m}+\sum_{i=2}^\infty\sum_{j=1}^i \sum_{U\in A_i\cap [1,n],|U|=j}-\log_m \frac{t_{x_U;i}}{t_{\phi;i}},
\end{align*}
where 
\begin{equation*}
    t_{x_U;i}:=\sum_{\ell_1,...,\ell_{i-j}=0}^{m-1}a_{x_{u_j},\ell_1}a_{\ell_1,\ell_2}\cdots a_{\ell_{i-j-1,i-j}}t_{x_U\ell_1\cdots \ell_{i-j};i},
\end{equation*}
and $t_{x_U\ell_1\cdots \ell_{i-j};i}$ is chosen from the definition of $\mu_i$ for all $i\geq 2$.
Then,
    \begin{equation}\label{eq 3.4-3}
        \begin{aligned}
    &-\log_m \mathbb{P}_{\infty}([x_1x_2\cdots x_n])\\
    =&\left[1-\frac{1}{p}-\frac{1}{q}+\frac{1}{\left(p,q\right)p_1q_1}+\sum_{i=2}^\infty\left(\sum_{j=1}^i P_{i,j}\right) \log_m t_{\phi;i}\right]n+E(n,x),
\end{aligned}
    \end{equation}
where 
\[
E(n,x)=\sum_{i=2}^\infty\sum_{j=1}^i \sum_{U\in A_i\cap [1,n],|U|=j}-\log_m t_{x_U;i}.
\]
Recalling the definition (\ref{Pij}) of $P_{i,j}$, we have that for $s\geq 2$,
\begin{equation}\label{E=E1+E2}
\begin{aligned}
E(n,x)=&\sum_{i=2}^\infty\sum_{j=1}^i \sum_{v_1,...,v_j=0}^{m-1} L(i,n,j,v_1\cdots v_j,x)(-\log_m t_{v_1\cdots v_j;i})\\
=&nE_1(s,n,x)+nE_2(s,n,x),
\end{aligned}
\end{equation}
where
\[
E_1(s,n,x)=\sum_{i=s+1}^\infty\sum_{j=1}^i P_{i,j} \sum_{v_1,...,v_j=0}^{m-1} \frac{L(i,n,j,v_1\cdots v_j,x)}{L(i,n,j)}(-\log_m t_{v_1\cdots v_j;i}),
\]
and
\[
E_2(s,n,x)=\sum_{i=2}^{s}\sum_{j=1}^i P_{i,j} \sum_{v_1,...,v_j=0}^{m-1} \frac{L(i,n,j,v_1\cdots v_j,x)}{L(i,n,j)}(-\log_m t_{v_1\cdots v_j;i}).
\]

Now, we estimate $E_1(s,n,x)$ and $E_2(s,n,x)$ in three steps:\\
(i) For $s\geq 2$ and for all $x\in X^{p,q;a,b}_A$,
\begin{equation}\label{eq 3.4-4}
    \begin{aligned}
    \left|E_1(s,n,x)\right|\leq & \sum_{i=s}^\infty\sum_{j=1}^i P_{i,j} \sum_{v_1,...,v_j=0}^{m-1} \frac{L(i,n,j,v_1\cdots v_j,x)}{L(i,n,j)}\log_m \left|A^{j-1}\right|\\
    = & \sum_{i=s}^\infty\sum_{j=1}^i P_{i,j} \log_m \left|A^{j-1}\right|
    \leq  \sum_{i=s}^\infty\sum_{j=1}^i P_{i,j} (j+2)\\
     =&\sum_{j=1}^s \sum_{i=s}^{\infty} P_{i,j}(j+2)+\sum_{j=s}^\infty \sum_{i=j}^\infty P_{i,j}(j+2)
    \leq \frac{M}{(p_1)^s},
\end{aligned}
\end{equation}
where $M>0$ is an absolute constant.\\
(ii) By selecting further subsequences and then applying diagonal process (from $s=2$ to $\infty$), we can find a subsequence $\{\ell_k\}_{k=1}^\infty$ of $\mathbb{N}$ such that the limits
\begin{align*}
   R_{v_1\cdots v_j;s}(x)&=\lim_{k\to\infty} \frac{L(s,\ell_k,j,v_1\cdots v_j,x)}{L(s,\ell_k,j)}~(\forall 1\leq j\leq s).
\end{align*}
all exist. Evidently, we have
\begin{equation}\label{R}
    R_{v_1\cdots v_s;i}(x)=\sum_{j=0}^{m-1} R_{v_1\cdots v_s j;i}(x),~( \forall i\geq 2 \mbox{ and } 1\leq s \leq i-1). 
\end{equation}
(iii) By (\ref{R}), for $i=2$ and for all $x\in X^{p,q;a,b}_A$, $R_{v_1;2}(x)=\sum_{b_2=0}^{m-1}R_{v_1b_2;2}(x)$. Hence,
\begin{align*}
    &\sum_{j=1}^2 P_{2,j} \sum_{v_1,...,v_j=0}^{m-1} \frac{L(2,n,j,v_1\cdots v_j,x)}{L(2,n,j)}(-\log_m t_{v_1\cdots v_j;2})\\
    \mathop{\to}_{n\to\infty}&P_{2,1} \sum_{v_1=0}^{m-1}\left(-R_{v_1;2}(x)\log_m t_{v_1;2}\right)+P_{2,2}\sum_{v_1=0}^{m-1}\left(-\sum_{b_2=0}^{m-1} R_{v_1b_2;2}(x)\log_m t_{v_1b_2;2}\right)\\
    =&P_{2,1} \sum_{v_1=0}^{m-1}\left(-R_{v_1;2}(x)\log_m (a_{v_1})^{\frac{P_{2,2}}{P_{2,1}+P_{2,2}}}\right)\\
    &+P_{2,2}\sum_{v_1=0}^{m-1}\left(-\sum_{b_2=0}^{m-1} R_{v_1b_2;2}(x)\log_m (a_{v_1})^{\frac{-P_{2,1}}{P_{2,1}+P_{2,2}}}\right)\\
    =&0,
\end{align*}
where the notation $\mathop{\to}_{n\to\infty}$ indicates the convergence along the sequence $\{n=\ell_k\}_{k=1}^\infty$.

By (\ref{R}), for $i=3$ and for all $x\in X^{p,q;a,b}_A$, $R_{v_1;3}(x)=\sum_{b_2=0}^{m-1}R_{v_1b_2;3}(x)$ and $R_{v_1b_2;3}(x)=\sum_{b_3=0}^{m-1}R_{v_1b_2b_3;3}(x)$. Hence,
\begin{align*}
     &\sum_{j=1}^3 P_{3,j} \sum_{v_1,...,v_j=0}^{m-1} \frac{L(3,n,j,v_1\cdots v_j,x)}{L(3,n,j)}(-\log_m t_{v_1\cdots v_j;3})\\
    \mathop{\to}_{n\to\infty}&P_{3,1} \sum_{v_1=0}^{m-1}\left(-R_{v_1;3}(x)\log_m t_{v_1;3}\right)
     +P_{3,2}\sum_{v_1=0}^{m-1}\left(-\sum_{b_2=0}^{m-1} R_{v_1b_2;3}(x)\log_m t_{v_1b_2;3}\right)\\
      &+P_{3,3}\sum_{v_1=0}^{m-1}\left(-\sum_{b_2,b_3=0}^{m-1} R_{v_1b_2b_3;2}(x)\log_m t_{v_1b_2b_3;3}\right)\\
      =&P_{3,1} \sum_{v_1=0}^{m-1}\left[-R_{v_1;3}(x)\log_m F(v_1)^{\frac{P_{3,2}+P_{3,3}}{\sum_{k=1}^3 P_{3,k}}}\right]\\
     &+P_{3,2}\sum_{v_1=0}^{m-1}\left[-\sum_{b_2=0}^{m-1} R_{v_1b_2;3}(x)\log_m (a_{b_2})^{\frac{P_{3,3}}{P_{3,2}+P_{3,3}}}F(v_1)^{\frac{-P_{3,1}}{\sum_{k=1}^3 P_{3,k}}}\right]\\
      &+P_{3,3}\sum_{v_1=0}^{m-1}\left[-\sum_{b_2,b_3=0}^{m-1} R_{v_1b_2b_3;3}(x)\log_m (a_{b_2})^{\frac{-P_{3,2}}{P_{3,2}+P_{3,3}}}F(v_1)^{\frac{-P_{3,1}}{\sum_{k=1}^3 P_{3,k}}}\right]\\
      =&0,
\end{align*}
where $F(v_1)=\sum_{j=0}^{m-1}a_{v_1,j}(a_j)^{\frac{P_{3,3}}{P_{3,2}+P_{3,3}}}$.

Similarly, for $i\geq 2$ and for all $x\in X^{p,q;a,b}_A$, we have
\begin{align*}
     \sum_{j=1}^i P_{i,j} \sum_{v_1,...,v_j=0}^{m-1} \frac{L(i,n,j,v_1\cdots v_j,x)}{L(i,n,j)}(-\log_m t_{v_1\cdots v_j;i}) \mathop{\to}_{n\to\infty} 0.
\end{align*}
Thus, for $s\geq 2$ and for all $x\in X^{p,q;a,b}_A$, 
\begin{equation}\label{eq 3.4-5}
    E_2(s,n,x) \mathop{\to}_{n\to\infty} 0.
\end{equation}

By (\ref{eq 3.4-3}), (\ref{E=E1+E2}), (\ref{eq 3.4-4}) and (\ref{eq 3.4-5}),
\begin{align*}
    &\left|-\frac{1}{\ell_k}\log_m \mathbb{P}_{\infty}([x_1x_2\cdots x_{\ell_k}])-\left[1-\frac{1}{p}-\frac{1}{q}+\frac{1}{\left(p,q\right)p_1q_1}+\sum_{i=2}^\infty\left(\sum_{j=1}^i P_{i,j}\right) \log_m t_{\phi;i}\right]\right|\\
    =&\left|\frac{E(\ell_k,x)}{\ell_k}\right|=\left|\frac{\ell_k E_1(s,\ell_k,x)+\ell_k E_2(s,\ell_k,x)}{\ell_k}\right| \leq \frac{M}{(p_1)^s}\quad(\forall s\geq 2). 
\end{align*}
Thus, letting $s\to\infty$, we derive our claim (\ref{claim2}).

Finally, by Lemmas \ref{lemma B} and \ref{lemma 3.1}, we have
 \begin{equation*}
	\dim_HX^{p,q;a,b}_A= 1-\frac{1}{p}-\frac{1}{q}+\frac{1}{\left(p,q\right)p_1q_1}+\sum_{i=2}^\infty\left(\sum_{j=1}^i P_{i,j}\right) \log_m t_{\phi;i}.
	\end{equation*}

\item[\bf (3)] By applying the decomposition outlined in Theorem \ref{Thm: 1} and then applying Lemma \ref{lemma 3.1}, we complete the proof in a similar way as that of Theorem \ref{Thm: 2}. 

\item[\bf (4)] First, applying Lemma \ref{lemma H}, we have
\begin{align*}
    \sum_{i=0}^{m-1}\left(a_i\right)^{\frac{p}{q}}&= \sum_{i=0}^{m-1}\left(a_i\right)^{\frac{p}{q}} (1)^{1-\frac{p}{q}}\leq \left(\sum_{i=0}^{m-1} a_i\right)^{\frac{p}{q}}\left(\sum_{i=0}^{m-1}1\right)^{1-\frac{p}{q}}\\
    &=\left(\sum_{i=0}^{m-1} a_i\right)^{\frac{p}{q}}\left(m\right)^{1-\frac{p}{q}},
\end{align*}
where the equality holds if $a_0=a_1=\cdots =a_{m-1}$. This implies that
\begin{align*}
    \log_m \sum_{i=0}^{m-1}\left(a_i\right)^{\frac{p}{q}} \leq \frac{p}{q}\log_m \sum_{i=0}^{m-1} a_i   +\left(1-\frac{p}{q}\right),
\end{align*}
where the equality holds if $a_0=a_1=\cdots =a_{m-1}$. Thus, 
\begin{align*}
\dim_HX^{p,q;a,b}_A&= 1-\frac{1}{p}-\frac{1}{q}+\frac{1}{p}\log_m \sum_{i=0}^{m-1}\left(a_i\right)^{\frac{p}{q}}\\
&\leq 1-\frac{1}{p}-\frac{1}{q}+\left(\frac{1}{p}-\frac{1}{q}\right)+\frac{1}{q}\log_m \sum_{i=0}^{m-1} a_i=\dim_MX^{p,q;a,b}_A,
\end{align*}
where the equality holds if $a_0=a_1=\cdots =a_{m-1}$.

Second, since
\begin{align*}
    t_{\phi;2}&=\sum_{j_1,j_2=0}^{m-1}a_{j_1,j_2}(a_{j_1})^{\frac{-P_{2,1}}{P_{2,1}+P_{2,2}}}=\sum_{j_1=0}^{m-1}(a_{j_1})^{\frac{P_{2,2}}{P_{2,1}+P_{2,2}}}(1)^{\frac{P_{2,1}}{P_{2,1}+P_{2,2}}}\\
    &\leq \left(\sum_{j_1=0}^{m-1}a_{j_1}\right)^{\frac{P_{2,2}}{P_{2,1}+P_{2,2}}}\left(\sum_{j_1=0}^{m-1}1\right)^{\frac{P_{2,1}}{P_{2,1}+P_{2,2}}}= \left(\sum_{j_1=0}^{m-1}a_{j_1}\right)^{\frac{P_{2,2}}{P_{2,1}+P_{2,2}}}(m)^{\frac{P_{2,1}}{P_{2,1}+P_{2,2}}},
\end{align*}
we have
\begin{equation*}
    \log_m t_{\phi;2} \leq \frac{P_{2,2}}{P_{2,1}+P_{2,2}}\log_m |A|+\frac{P_{2,1}}{P_{2,1}+P_{2,2}},
\end{equation*}
where the equality holds if $a_0=\cdots =a_{m-1}$. 

Again, by Lemma \ref{lemma H}, we have
\begin{align*}
    \sum_{j=0}^{m-1}a_{r_1,j} (a_{j})^{\frac{P_{3,3}}{P_{3,2}+P_{3,3}}}&= \sum_{j=0}^{m-1}a_{r_1,j} (a_{j})^{\frac{P_{3,3}}{P_{3,2}+P_{3,3}}} (1)^{\frac{P_{3,2}}{P_{3,2}+P_{3,3}}}\\
    &\leq \left(\sum_{j=0}^{m-1}a_{r_1,j} (a_{j})\right)^{\frac{P_{3,3}}{P_{3,2}+P_{3,3}}} \left(\sum_{j=0}^{m-1}a_{r_1,j}\right)^{\frac{P_{3,2}}{P_{3,2}+P_{3,3}}}\\
    &\leq \left(\sum_{j=0}^{m-1}a_{r_1,j} (a_{j})\right)^{\frac{P_{3,3}}{P_{3,2}+P_{3,3}}} \left(a_{r_1}\right)^{\frac{P_{3,2}}{P_{3,2}+P_{3,3}}},
\end{align*}
where all equalities hold if all $a_j$ are equal for every $j$ with $a_{r_1,j}=1$. 

For $i=3$,
\begin{align*} 
    t_{\phi;3}= &\sum_{r_1,r_2}^{m-1}a_{r_1,r_2}(a_{r_2})^{\frac{P_{3,3}}{P_{3,2}+P_{3,3}}}\left(\sum_{j=0}^{m-1}a_{r_1,j}(a_j)^{\frac{P_{3,3}}{P_{3,2}+P_{3,3}}}\right)^{\frac{-P_{3,1}}{\sum_{k=1}^3 P_{3,k}}}\\
    \leq &\sum_{r_1=0}^\infty\left[\left(\sum_{j=0}^{m-1}a_{r_1,j} (a_{j})\right)^{\frac{P_{3,3}}{P_{3,2}+P_{3,3}}} \left(a_{r_1}\right)^{\frac{P_{3,2}}{P_{3,2}+P_{3,3}}}\right]^{1+\frac{-P_{3,1}}{\sum_{k=1}^3 P_{3,k}}}\\
    = &\sum_{r_1=0}^\infty\left[\left(\sum_{j=0}^{m-1}a_{r_1,j} (a_{j})\right)^{\frac{P_{3,3}}{\sum_{k=1}^3 P_{3,k}}} \left(a_{r_1}\right)^{\frac{P_{3,2}}{\sum_{k=1}^3 P_{3,k}}} (1)^{\frac{P_{3,1}}{\sum_{k=1}^3 P_{3,k}}}\right]\\
    \leq &\left(\sum_{r_1,j=0}^{m-1}a_{r_1,j} (a_{j})\right)^{\frac{P_{3,3}}{\sum_{k=1}^3 P_{3,k}}} \left(\sum_{r_1=0}^{m-1}a_{r_1}\right)^{\frac{P_{3,2}}{\sum_{k=1}^3 P_{3,k}}} \left(m\right)^{\frac{P_{3,1}}{\sum_{k=1}^3 P_{3,k}}}\\
    =&|A^2|^{\frac{P_{3,3}}{\sum_{k=1}^3 P_{3,k}}} |A|^{\frac{P_{3,2}}{\sum_{k=1}^3 P_{3,k}}}  (m)^{\frac{P_{3,1}}{\sum_{k=1}^3 P_{3,k}}},
\end{align*}
where all equalities hold if all $a_j$ are equal for every $0\leq j \leq m-1$.

Therefore, we have 
\begin{equation*}
\log_m t_{\phi;3}\leq \frac{P_{3,3}\log_m \left|A^2\right| }{\sum_{k=1}^3 P_{3,k}} +\frac{P_{3,2}\log_m \left|A\right|}{\sum_{k=1}^3 P_{3,k}}  +\frac{P_{3,1}}{\sum_{k=1}^3 P_{3,k}},
\end{equation*}
where the equality holds if $a_0=a_1=\cdots =a_{m-1}$.

Similarly, we deduce that for $i\geq 2$,
\begin{equation}\label{eq 3.4-1}
    \log_m t_{\phi;i}\leq \frac{\sum_{k=1}^{i} P_{i,k} \log_m \left|A^{k-1}\right|}{\sum_{j=1}^i P_{i,j}},
\end{equation}
where the equality holds if $a_0=a_1=\cdots =a_{m-1}$.

Thus, by (\ref{eq 3.4-1}), we have
\begin{align*}
&\dim_HX^{p,q;a,b}_A\\
=& 1-\frac{1}{p}-\frac{1}{q}+\frac{1}{\left(p,q\right)p_1q_1}+\sum_{i=2}^\infty\left(\sum_{j=1}^i P_{i,j}\right) \log_m t_{\phi;i}\\
\leq& 1-\frac{1}{p}-\frac{1}{q}+\frac{1}{\left(p,q\right)p_1q_1}+\sum_{i=2}^\infty\sum_{k=1}^{i} P_{i,k} \log_m \left|A^{k-1}\right|\\
=& 1-\frac{1}{p}-\frac{1}{q}+\frac{1}{\left(p,q\right)p_1q_1}+\sum_{i=2}^\infty P_{i,1}+\sum_{k=2}^{\infty}\sum_{i=k}^\infty P_{i,k} \log_m \left|A^{k-1}\right|\\
=&1-\frac{1}{p}-\frac{1}{q}+\frac{1}{(p,q)p_1q_1}+\left(1-\frac{1}{q_1}\right)\left(\frac{1}{p}-\frac{1}{q}\right)+\frac{(q_1-1)^2}{(p,q)}\sum_{k=2}^\infty \frac{\log_m |A^{k-1}|}{(q_1)^{k+1}}\\
=&\dim_MX^{p,q;a,b}_A,
\end{align*}
where the equality holds if $a_0=a_1=\cdots =a_{m-1}$.
\end{proof}	

\section{Higher-order interaction}\label{sec4}
In this section, we consider the interactions of higher-order multiplicative constraints. For $\ell\geq 1$, let $\mathcal{F}$ be a subset of $\Sigma_m^{\ell+2}$ and let $f_1(k),...,f_\ell(k):\mathbb{N}\to \mathbb{N}$ be functions. Define
\begin{equation*}
    X_{\mathcal{F}}^{p,q;a,b}(f_1,...,f_\ell)=\left\{(x_i)_{i=1}^\infty\in \Sigma_m^{\mathbb{N}}:x_{pk+a}x_{qk+b}x_{f_1(k)}\cdots x_{f_\ell(k)} \notin \mathcal{F},\forall k\geq 1\right\}
\end{equation*}
and when $\ell=0$, define
\begin{equation*}
    X_{\mathcal{F}}^{p,q;a,b}=\left\{(x_i)_{i=1}^\infty\in \Sigma_m^{\mathbb{N}}:x_{pk+a}x_{qk+b} \notin \mathcal{F},\forall k\geq 1\right\}.
\end{equation*}

\begin{theorem}
 For all $1\leq i\leq \ell$, if $f_i(k)=\Theta(k^{s_i})$ with $s_i>1$, then
	\[
	\dim_{M}X_{\mathcal{F}}^{p,q;a,b}(f_1,...,f_\ell)=\dim_{M}X_{\mathcal{F}}^{p,q;a,b},
	\]  
 and
 \[
	\dim_{H}X_{\mathcal{F}}^{p,q;a,b}(f_1,...,f_\ell)=\dim_{H}X_{\mathcal{F}}^{p,q;a,b}. 
	\] 
\end{theorem}
\begin{proof}
	For $n\geq 1$, the decomposition of $\{1,...,n\}$ with respect to the constraints $(pk+a,qk+b)$ is a disjoint union of one dimensional lattices (see Theorem \ref{Thm: 1}). Observe that the set $\left\{k\in \{1,...,n\}: f_i(k) \leq n\right\}\subseteq\{1,...,n\}$ is influenced by the higher-order constraint. This implies that the constraint $f_i(k)$ induces at most
 \[
  2\#\left\{k\in\{1,...,n\} : f_i(k) \leq n\right\}
 \]
 many disjoint lattices of decomposition of $\{1,...,n\}$ with respect to the constraints $(pk+a,qk+b)$. 

For any disjoint lattice of decomposition of $\{1,...,n\}$ with respect to the constraints $(pk+a,qk+b)$, the cardinality of the lattice is at most $\log_{q_1}n$, where $q_1=\frac{q}{\gcd(p,q)}$, and 
 \begin{align*}
     &\frac{ 2\#\left\{k\in\{1,...,n\} : f_i(k) \leq n\right\} \times \log_{q_1}n}{n}\\
     \leq &\frac{ 2\#\left\{k\in\{1,...,n\} : k^{s_i} \leq n\right\} \times \log_{q_1}n}{n}\\
     \leq &\frac{ 2 n^{\frac{1}{s_i}} \times \log_{q_1}n}{n}\to 0\mbox{ as }n\to\infty \mbox{ (since }s_i>1).
 \end{align*}
Then, a similar estimate gives that
  \begin{align*}
     &\frac{ 2\#\cup_{i=1}^\ell\left\{k\in\{1,...,n\} : f_i(k) \leq n\right\} \times \log_{q_1}n}{n}\leq \frac{ 2 \sum_{i=1}^\ell n^{\frac{1}{s_i}} \times \log_{q_1}n}{n}\mathop{\to}_{n\to\infty} 0.
 \end{align*}
	Thus, the constraints $f_1(k)$, ...,$f_\ell(k)$ can be omitted. The proof is complete.
\end{proof}

\bibliographystyle{amsplain}
\bibliography{ban}
\end{document}